\documentclass[12pt]{amsart}

\usepackage[top=100pt,bottom=60pt,left=96pt,right=92pt]{geometry}

\usepackage{hyperref}
\usepackage{amsmath,amsthm,amssymb,amsfonts,enumerate,xcolor,layout,tikz,subcaption,enumitem,fullpage, verbatim, colordvi}
\usetikzlibrary{calc}   % gives us abs(), arithmetic in coordinates
\usepackage{stmaryrd}
\usepackage{graphicx}
\usepackage{subcaption}
\newtheorem{theorem}{Theorem}[section]
\newtheorem*{theorem*}{Theorem}
\newtheorem{proposition}[theorem]{Proposition}
\newtheorem{corollary}[theorem]{Corollary}
\newtheorem{lemma}[theorem]{Lemma}

\theoremstyle{definition}

\newtheorem{remark}[theorem]{Remark}

\theoremstyle{plain}

\definecolor{linkblue}{rgb}{0,0,.6}
\definecolor{citered}{rgb}{.7,0,0}
\hypersetup{colorlinks =true, linkcolor=linkblue, citecolor = citered, urlcolor=linkblue}

\def\C{{\mathbb C}}
\def\N{{\mathbb N}}

\def\R{{\mathbb R}}

\def\Z{{\mathbb Z}}

\def\GL{\textrm{GL}}
\newcommand{\medcup}{\textstyle{\bigcup}}
\newcommand{\medcap}{\textstyle{\bigcap}}

\newcommand{\De}{\Delta}
\newcommand{\om}{\omega}
\newcommand{\fh}{\mathfrak{h}}
\newcommand{\ft}{\mathfrak{t}}

\newcommand{\relint}{\operatorname{relint}}

\newcommand{\short}{\textrm{short}}

\newcommand{\Ext}{\operatorname{Ext}}

\newcommand{\purple}[1]{\textbf{\purple{#1}}}

\title{Skeletons and Toric Extensions of Maximally Short Complexity One Spaces}
\author{Yichen Liu}
\date{}

\begin{document}
\begin{abstract}
Complexity one $T$-spaces are Hamiltonian $T$-spaces $(M,\om,\Phi)$ such that $\frac{1}{2}\dim M -\dim T=1$. The skeleton of a complexity one $T$-space is an important invariant in the classification and encodes the information about non-generic orbits. In this paper, we prove that the moment image of the skeleton of a compact, connected maximally short complexity one $T$-space, which is in fact a GKM space, is connected. The proof relies on the well-known fact that each connected component of regular values of a proper moment map is a convex locally polyhedral set. We also gave an elementary proof of that fact along the way. Then we use the connectedness result to estimate the number of symplectic toric $(T \times S^1)$-manifolds whose underlying complexity one $T$-space is the same as the given maximally short complexity one $T$-space.
\end{abstract}
\maketitle

\section{Introduction}
Let $T$ be a torus with Lie algebra $\mathfrak{t}$ and dual space $\ft^*$. Recall that we say $T$ acts on a symplectic manifold $(M,\omega)$ \textit{in a Hamiltonian fashion} if there exists a $T$-invariant map $\Phi: M \to \ft^*$ such that $d \langle \Phi, X \rangle = \omega(X^\#, \cdot)$, where $X^\#$ is the fundamental vector field on $M$ corresponding to $X$. We call the tuple $(M,\om,\Phi)$ a {\bf Hamiltonian $T$-space}. If not otherwise stated, we always assume the manifold is connected and the action is effective in this paper. Our first main result is the convexity of the regular value chambers. This result is well-known in the compact case (cf.~\cite[Example 3.7]{GLS}). In this paper, we provide an elementary proof by only assuming the moment map to be proper.

\begin{theorem}\label{thm:convexity-of-chambers}
    Let $(M,\omega,\Phi)$ be a connected Hamiltonian $T$-space with proper moment map $\Phi: M \to \mathfrak{t}^*$. Let $\mathcal{R}$ denote the set of regular values of $\Phi$ inside $\Phi(M)$. Then the closure  $\overline{\mathcal{C}}$ of each connected component $\mathcal{C}$ of $\mathcal{R}$ is a convex rational locally polyhedral set, and $\Phi(M)$ is the union of all such $\overline{\mathcal C}$.
\end{theorem}

The {\bf complexity} of a Hamiltonian $T$-space $(M,\om,\Phi)$ is $\frac{1}{2} \dim M - \dim T$; it is half the dimension of the reduced space $\Phi^{-1}(\alpha)/T$ at a regular value $\alpha \in \Phi(M)$. Hamiltonian $T$-spaces of complexity zero are known as symplectic toric manifolds. Delzant~\cite{De88} classified compact symplectic toric manifolds by their moment image. Later, Lerman-Tolman~\cite{LT97} and Karshon-Lerman~\cite{KL} respectively generalized his classification to symplectic toric orbifolds and symplectic toric manifolds with proper moment map. Complexity one spaces were first studied in dimension four by Karshon~\cite{Karshon99}. Later, in a series of papers~\cite{KT01,KT03,KT14}, Karshon and Tolman classified tall complexity one spaces with proper moment map in any dimensions. These spaces also demonstrate new phenomena: Tolman~\cite{Tolman98} constructed a complexity one space that admits no invariant K\"ahler structure, while all symplectic toric manifolds are K\"ahler~\cite{De88}. 

Fix a complexity one $T$-space $(M,\om,\Phi)$ with proper moment map, we say a point $p \in M$ is {\bf tall} if $\Phi^{-1}(\Phi(p))/T$ is not a singleton, and {\bf short} otherwise. Correspondingly, a moment value $\alpha \in \Phi(M)$ is tall if $\Phi^{-1}(\alpha)/T$ is not a singleton, and short otherwise. We denote the set of short moment values by $\Phi(M)_\short$. A complexity one $T$-space $(M,\om,\Phi)$ is {\bf tall} if every point in $M$ is tall; it is {\bf maximally short} if $ \Phi(M)_\short = \partial \Phi(M)$. Karshon and Tolman~\cite[Lemma 5.4]{KT01} proved that $\Phi(M)_\short$ is a subset of $\partial \Phi(M)$, so this justifies the name ``maximally short''.

An orbit $\mathcal{O}_p \subset M$ is {\bf exceptional} if every
nearby orbit in the moment fiber $\Phi^{-1}(\Phi(p))$ has strictly smaller stabilizer. Correspondingly, a point in an exceptional orbit is called an exceptional point.
The {\bf skeleton} $\Sigma$ of $M$ is the set of tall exceptional orbits (an orbit is tall if any and hence all points in the orbit are tall), considered as a subspace of $M/T.$ The {\bf $m$-skeleton} $\Sigma_m$ is the set of tall exceptional orbits that are at most $m$-dimensional. 
Given $k \in \N$, we say that the skeleton of $M$ is {\bf $\mathbf k$-colorable} if the skeleton is the disjoint union of $k$ (possibly empty) clopen\footnote{More explicitly, each subset is closed and open in the subset topology on the (one-)skeleton.} subsets
so that the orbital moment map\footnote{With the slight abuse of notation, we denote both the moment map and the orbital moment map (the map induced by the moment map on the orbit space) by $\Phi$.} $\Phi \colon M/T \to \ft^*$
restricts to an injection on each subset.
In~\cite{Liu}, we study properties of the skeleton of a compact, connected tall complexity one $T$-space. In this paper, we focus on another extremal case, maximally short complexity one $T$-spaces.

\begin{theorem}\label{thm:skeleton-connected}
     Let $(M,\omega,\Phi)$ be a compact, connected, maximally short complexity one $T$-space. Let $\Sigma \subset M/T$ be its skeleton. If $\dim M \geq 6$, then $\Phi(\Sigma)$ is connected.
\end{theorem}

This theorem is not true in dimension four, because there exist Hamiltonian $S^1$-spaces whose orbits are either isolated fixed points or free orbits. It is not true either if the complexity one $T$-space is tall (see Figure~\ref{fig:counterexample}). Moreover, Theorem~\ref{thm:skeleton-connected} gives a bound for the number of symplectic toric $(T \times S^1)$-manifolds whose underlying complexity one $T$-space is the same as the given maximally short complexity one $T$-space.

\begin{figure}
    \centering
    \begin{subfigure}{.45\textwidth}
    \centering
            \begin{tikzpicture}
      % Vertices
\coordinate (A) at (0,0,0);
\coordinate (B) at (0,0,1);
\coordinate (C) at (0,1,1);
\coordinate (D) at (1,1,2);
\coordinate (E) at (1,0,2);
\coordinate (F) at (2,0,2);
\coordinate (G) at (2,1,2);
\coordinate (H) at (3,0,1);
\coordinate (I) at (3,1,1);
\coordinate (J) at (3,0,0);
\coordinate (K) at (3,1,0);
\coordinate (L) at (0,1,0);

% Bottom face
\draw (A)--(J)--(K)--(L)--cycle;

% Top face
\draw (B)--(C)--(D)--(G)--(I)--(H)--(F)--(E)--cycle;

% Vertical edges
\draw (A)--(B);
\draw (L)--(C);
\draw (J)--(H);
\draw (K)--(I);

% Upper ridge
\draw[thick,blue] (D)--(E);
\draw (E)--(F);
\draw[thick,red] (F)--(G);
    \end{tikzpicture}
    \end{subfigure}
    \begin{subfigure}{.45\textwidth}
    \centering
    \begin{tikzpicture}
      \draw (7,0) -- (10,0) -- (10,2) -- (7,2) -- (7,0);
      \draw[thick,blue] (8,0) -- (8,2);
      \draw[thick,red] (9,0) -- (9,2);
    \end{tikzpicture}
        \end{subfigure}
    \caption{Consider the six dimensional symplectic toric manifold with the Delzant polytope given on the left (we orient the axes so that the rectangle on the back lies on the $xy$-plane). If we restrict the $T^3$-action to $T^2 \times \{1\}$, then we get a tall complexity one $T^2$-space and the moment image of the skeleton has two connected components (the red and blue line segments).}
    \label{fig:counterexample}
\end{figure}
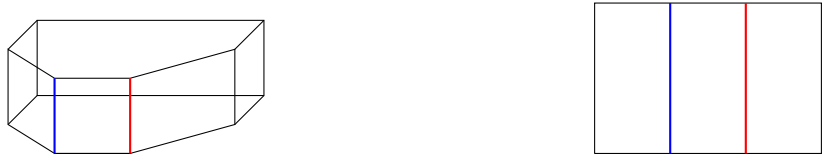

Fix a complexity one $T$-space $(M,\om,\Phi)$, a function $f: M \to \R$ is a {\bf toric extension} of $(M,\om,\Phi)$ if $(M,\om, (\Phi,f))$ is a symplectic toric $(T\times S^1)$-manifold. Two functions $f,g$ are equivalent if $f-g$ is a constant. Let $\Ext(M,\omega,\Phi)$ denote the set of equivalent classes of toric extensions of the complexity one $T$-space $(M,\om,\Phi)$. Notice that $\ell\times \Z_2$ acts on this set by
$(A,-1).f = -(f+\langle\Phi, A\rangle$), here $\ell \subset \ft$ is the integral lattice and $\Z_2 =\{-1,1\}$. In~\cite{LPT}, we gave a sufficient and necessary condition for when $\Ext(M,\omega,\Phi)$ is nonempty. Moreover, we noticed that the total number of extensions is closely related to the number of $2$-colorings (i.e. different decomposition of the skeleton into $2$ disjoint clopen subsets so that the orbital moment map is injective on each subset) of the skeleton. Theorem~\ref{thm:skeleton-connected} helps us determine the size of $\Ext(M,\omega,\Phi)$.

\begin{theorem}\label{thm:max-short-bound}
    Let $(M,\omega,\Phi)$ be a compact, connected, maximally short complexity one $T$-space of dimension at least $6$. If $\Ext(M,\omega,\Phi)$ is nonempty, then $\Ext(M,\omega,\Phi) \cong \ell \times \Z_2$.
\end{theorem}

The structure of the paper is as follows. In Section~\ref{sec:background}, we review some properties of complexity one spaces and some facts about Duistermaat-Heckman measures. In Section~\ref{sec:convexity}, we first prove a general convexity result and use it to prove an upgraded version (Theorem~\ref{thm:convexity-regular-values}) of Theorem~\ref{thm:convexity-of-chambers}. In Section~\ref{sec:connectedness}, we use Theorem~\ref{thm:convexity-of-chambers} and Duistermaat-Heckman measures to prove Theorem~\ref{thm:skeleton-connected}. In fact, we prove a stronger result Theorem~\ref{thm:n-1-skeleton} which states that the moment image of the set of exceptional orbits with codimension at least $1$ is connected. Then we use Theorem~\ref{thm:skeleton-connected} to prove Theorem~\ref{thm:max-short-bound}.

{\bf Acknowledgements.} The author would like to thank Susan Tolman for her patient guidance and invaluable suggestions.

\section{Background}\label{sec:background}
In this section, we prove some useful facts about maximally short complexity one spaces and review important results about Duistermaat-Heckman measures.

Let $(M,\omega,\Phi)$ be a Hamiltonian $T$-space.
Recall that the symplectic slice at $p \in M$ is the symplectic vector space \[(T_p \mathcal{O})^\omega / (T_p \mathcal{O} \cap (T_p \mathcal{O})^\omega)\] where $\mathcal{O}$ is the $T$-orbit of $p$. Let $H \subset T$ be the stabilizer group of $p$. The isotropy representation of $H$ on $T_p M$ induces a representation on the quotient space, called the {\bf slice representation.} Since the symplectic slice inherits the complex structure from $T_pM$, the slice representation decomposes into one-dimensional complex irreducible representations each of which is determined by a weight vector $\eta_i \in \mathfrak{h}^*$. Under the identification $(T_p \mathcal{O})^\omega / (T_p \mathcal{O} \cap (T_p \mathcal{O})^\omega) \cong \C^k$, we can express the slice representation as a group homomorphism $\rho: H \to (S^1)^k$ such that $h.(z_1,\ldots,z_k) = (\rho_1(h)z_1,\ldots,\rho_k(h)z_k)$, where the differential of $\rho_i$ is $\eta_i$.

Once and for all, we fix an inner product on $\mathfrak{t}$ and identify $\mathfrak{h}^*$ as a subspace of $\mathfrak{t}^*$.
Consider the space
$Y = T \times_H \fh^\circ \times \C^k$
with $T$ acting on the left. The action of $T$ is Hamiltonian,
and there is a moment map
$\Phi_Y([t,\nu,z]) = \frac{1}{2} \sum_{j=1}^k |z_j|^2 \eta_j + \nu.$
We will call this the \emph{local model} because of the following local normal form by Guillemin-Sternberg \cite{GS84} and Marle~\cite{Marle}.

\begin{theorem}[Local normal form]\label{thm:lnf}
Let $(M,\omega,\Phi)$ be a Hamiltonian $T$-space. Given a point $p \in M$ with stabilizer group $H \subset T$ and slice representation $\rho$, there exists a neighborhood of the orbit $T \cdot p$ that is equivariantly symplectomorphic to a neighborhood of the orbit $\{[t,0,0]\}$ in the model $Y:= T \times_H \fh^\circ \times \C^k$.
\end{theorem}

The next lemma establishes that the moment image of the local model agrees locally with $\Phi(M)$. It is a special case of \cite[Theorem 6.5]{Sjamaar}; see also \cite[Theorems 1.2, 4.3, 6.1 and 6.2]{LMTW}. 

\begin{lemma}\label{lem:local-cone}
Let  $(M,\omega,\Phi)$ be a connected Hamiltonian $T$-space so that $\Phi$ is proper as a map to a convex open subset of $\ft^*$.
Given $p \in M$,  let $Y_p$ be the local model associated to $p$ with moment map $\Phi_p$, normalized so that $\Phi_p([1,0,0]) = \Phi(p)$.
Then there exists an open neighborhood $U$ of $\Phi(p)$ such that
$\Phi(M) \cap U = \Phi_p(Y_p) \cap U$.
\end{lemma}

We call a model $Y$ a \textit{complexity one local model} if $(Y,\omega,\Phi_Y)$ is a complexity one space. In this case, $Y =  T \times_H \fh^\circ \times \C^{h+1}$, where $h =\dim H$.
The next lemma follows immediately from \cite[Lemmas 5.4 and 5.8]{KT01} and it plays an essential role in this section.

\begin{lemma}\label{lem:defining-monomial}
    Let $Y =  T \times_H \fh^\circ \times \C^{h+1}$ be a complexity one local model with moment map $\Phi_Y([t,\nu,z]) = \frac{1}{2} \sum_{j=0}^{h} |z_j|^2 \eta_j + \nu$. There exists a (unique up to sign) vector $\xi  \in \Z^{h+1}$ such that the following sequence is exact:
\begin{equation}\label{ses} 1 \to H \xrightarrow{\rho} (S^1)^{h+1} \xrightarrow{P} S^1 \to 1,
\end{equation}
where $\rho:H \to (S^1)^{h+1}$ is the slice representation at $[1,0,0]$ and $P(t_0,...,t_{h}) = \prod_{j=0}^{h} t_j^{\xi_j}$. Moreover, $\sum a_j \eta_j=0$ if and only if $a = \lambda \xi$ for some $\lambda \in \R$.  Finally, $Y$ is tall if and only if $\xi \in \Z^{h+1}_{\geq 0}$ or $\xi \in \Z^{h+1}_{\leq 0}$.
\end{lemma}

The monomial $P$ is called the {\bf defining monomial of the complexity one local model $Y$.}
We now give a criterion for exceptional points in terms of the exponents in the defining monomials.

\begin{lemma}\label{lem:tall/exc-criteria}
    Let $Y = T \times_H \fh^\circ \times \C^{h+1}$ be a complexity one local model with moment map $\Phi_Y([t,\nu,z]) = \frac{1}{2} \sum_{k=0}^{h} |z_k|^2 \eta_k + \nu $ and let $\xi_0,\ldots,\xi_h$ be the exponents of the defining monomial.
\begin{enumerate}
\item A point $[t, \nu,z] \in Y$ is tall if and only if $\xi_i \xi_j \geq 0$ for all $i, j \in \{0,\dots,h\}$ such that $z_i=z_j=0$.
\item  A point $[t, \nu, z] \in Y$ is exceptional if and only if
$\sum_{z_k = 0} |\xi_k| > 1$, where the sum is over all $k \in \{0,\dots,h\}$ such that $z_k = 0$.
\end{enumerate}
\end{lemma}

\begin{proof}
    (1) By definition, $[t, \nu,z]$ is tall if and only if there exists another orbit in the same moment fiber.
    By Lemma~\ref{lem:defining-monomial}, a point $[t',\nu',w]$ is in the same fiber as $[t,\nu,z]$ if and only if $\nu' = \nu$ and there exists $s \in \R$ such that $|w_k|^2 -|z_k|^2 = s \xi_k$ for all $k=0,...,h$. Since $\xi_0,\ldots,\xi_h$ are not all zero, $|w_k|= |z_k|$ for all $k =0,\ldots,h$ if and only if $s=0$. Hence, $[t, \nu,z]$ is tall if and only if there exists $s \neq 0$ such that $|z_k|^2 +s \xi_k \geq 0$ for all $k =0,\ldots,h$.
    
    If $\xi_i \xi_j \geq 0$ for all $i, j \in \{0,\dots,h\}$ such that $z_i=z_j=0$, then we may assume that $\xi_i \geq0$ for all $i \in \{0,\dots,h\}$ such that $z_i=0$. Then there exists $s >0$ such that $|z_k|^2+ s\xi_k\geq 0$ for all $k \in \{0,\dots,h\}$, so $[t,\nu,z]$ is tall. If $[t,\nu,z]$ is tall, then there exists $s \neq 0$ such that $|z_k|^2 +s \xi_k \geq 0$ for all $k =0,\ldots,h$. In particular, for all $i,j \in \{0,\dots,h\}$ such that $z_i=z_j=0$, $s\xi_i \geq 0$ and $s \xi_j \geq 0$, so $\xi_i\xi_j \geq 0$.
    
    (2) To prove the criterion for the exceptional points, we compare the stabilizer groups of points in the same fiber. If we identify $H$ with its image under the embedding $\rho$ to $(S^1)^{h+1}$, the stabilizer group of $[t,\nu,z]$ is the intersection of $H$ with $\{(t_0,...,t_{h}): t_i =1 \textrm{ if } z_i \neq 0\}$. By the short exact sequence~\eqref{ses}, $H = \ker P$, so 
    \[\textrm{stab}_{[t,\nu,z]} \cong \{(t_0,...,t_{h}): \prod_{z_i=0} t_i^{\xi_i} = 1 \textrm{ and } t_i =1 \textrm{ if } z_i \neq 0 \}.\]

    If $[t,\nu,z]$ is a tall point, there exists a point $[t',\nu,w]$ and an $s \in \R \setminus\{0\}$ such that $|w_k|^2 -|z_k|^2 = s \xi_k$ for all $k=0,...,h$. Moreover, we can choose $s$ sufficiently small such that $|w_k| = |z_k|$ whenever $\xi_k=0$ and $w_k \neq 0$ whenever $\xi_k \neq 0$. Hence, $w_k=0$ if and only if $z_k=0$ and $\xi_k=0$, so $\prod_{w_i = 0} t_i^{\xi_i} =1$.
    \[\textrm{stab}_{[t',\nu,w]} \cong \{(t_0,...,t_{h}): t_i =1 \textrm{ if } z_i \neq 0 \textrm{ or } \xi_i \neq 0 \}.\]

    $[t,\nu,z]$ is exceptional if and only if $\textrm{stab}_{[t'\nu,w]} \subsetneq \textrm{stab}_{[t,\nu,z]}$. It is straightforward to check that $\textrm{stab}_{[t',w,\nu]} \subseteq \textrm{stab}_{[t,z,\nu]}$. The equality holds if and only if $\prod_{z_i=0} t_i^{\xi_i} = 1 $ implies $t_i =1$ for all $i$ such that $\xi_i \neq 0$. Notice that  $\prod_{z_i=0} t_i^{\xi_i} =  \prod_{z_i=0,\xi_i \neq0} t_i^{\xi_i}$, so the two stabilizer groups are equal if and only if one of the following holds:
    \begin{itemize}
        \item $z_i=0$ for all $i$ such that $\xi_i =0$, or
        \item there exists a unique $i \in \{0,\ldots,h\}$ such that $|\xi_i|=1$ and $z_i=0$.
    \end{itemize}
    
    Hence, $[t,\nu,z]$ is exceptional if and only if there exist $i,j$ such that $z_i=z_j=0$ and $\xi_i,\xi_j \neq 0$ or there exists $i$ such that $z_i=0$ and $|\xi_i|>1$. Therefore, $[t,\nu,z]$ is exceptional if and only if $\sum_{z_k=0} |\xi_k| >1$.

    If $[t,\nu,z]$ is a short point, then (1) implies that there exist $i,j \in \{0,\ldots,h\}$ such that $z_i=z_j=0$ and $\xi_i\xi_j<0$, so $\sum_{z_k=0} |\xi_k| \geq |\xi_i|+|\xi_j| \geq 2$. This concludes the proof.
\end{proof}

When the complexity one $T$-space is maximally short, the isotropy weights at any point are in general position.

\begin{lemma}\label{lem:maximally-short-weights}
    Let $(M,\omega,\Phi)$ be a maximally short complexity one $T$-space. Let $Y = T \times_H \fh^\circ \times \C^{h+1}$ be the local model associated to a point $p \in M$ with moment map $\Phi_Y([t,\nu,z]) = \Phi(p) + \frac{1}{2}\sum_{j=0}^{h} |z_j|^2 \eta_j + \nu$. Let $\xi_0,\ldots,\xi_h$ be the exponents of the defining monomial of $Y$. Then $\xi_0,\ldots,\xi_h$ are all nonzero. Equivalently, any $h$ weights are linearly independent.
\end{lemma}

\begin{proof}
By Lemma~\ref{lem:defining-monomial}, $\sum a_i\eta_i =0$ if and only if $a = \lambda \xi$ for some $\lambda\in \R$, so $\xi_0,\ldots,\xi_h$ are all nonzero if and only if any $h$ weights are linearly independent.

If $p$ is a tall point, then by definition $\Phi(p)\in \Phi(M)^\circ$, so Lemma~\ref{lem:local-cone} implies that $\Phi_Y$ is surjective. Hence, $\fh^* = \sum_{j=0}^h \R_{\geq 0} \eta_j$. In particular, there exist $a_j \geq 0$ such that $\sum a_j \eta_j = \sum -\eta_j$, so $\sum (a_j+1)\eta_j=0$. By Lemma~\ref{lem:defining-monomial}, $\xi_0,\ldots,\xi_h$ are all nonzero. (c.f.~\cite[Lemma 5.2]{KT01})

If $p$ is a short point, then $\Phi(p) \in \partial\Phi(M)$, so Lemma~\ref{lem:local-cone} implies that $\partial \Phi_Y(Y)$ agrees with $\partial \Phi(M)$ locally around $\Phi(p)$, hence any $[t,\nu,z] \in \Phi_Y^{-1}(\partial \Phi_Y(Y))$ is a short point. Let $C_H: = \Phi(p) + \sum_{j=0}^h \R_{\geq 0} \eta_j$. Then $\partial \Phi_Y(Y) = \partial C_H \times \fh^\circ$. By~\cite[Lemma 5.4]{KT01}, $\Phi_Y$ is proper, so $C_H$ is a proper cone. There are two cases: (i) each of $\eta_j$ directs an edge of $C_H$ or (ii) without loss of generality, $\eta_0$ does not direct an edge. In case (i), by definition of an edge, any $h$ weights are linearly independent. In case (ii), $C_H$ is a simple cone, so for any $i\in\{1,\ldots,h\}$, $\Phi(p)+\sum_{j=1, j\neq i}^h \R_{\geq 0} \eta_j$ is a facet of $C_H$. Hence, for $z \in \C^{h+1}$ with $z_0=z_i=0$ and $z_j = 1$ for all $j \neq 0,i$, $\Phi_Y([1,0,z]) \in \partial \Phi_Y(Y)$, so $[1,0,z]$ is a short point. By item (1) of Lemma~\ref{lem:tall/exc-criteria}, $\xi_0\xi_i<0$. Hence, $\xi_0,\ldots,\xi_h$ are all nonzero.
\end{proof}

\begin{remark}
    In particular, Lemma~\ref{lem:maximally-short-weights} implies that any maximally short complexity one $T$-space with dimension at least $6$ are GKM $T$-manifolds~\cite{GKM98}. In fact, in dimension six, a complexity one $T^2$-space is GKM if and only if it is maximally short. This is not true in higher dimensions. For example, take any maximally short complexity one $T^2$-space $(M,\omega,\Phi)$ and consider its product with $(S^2,d\theta \wedge dh, h)$. One can verify that the product is still a GKM manifold, but not maximally short anymore, because for any $\alpha \in \Phi(M)^\circ$, the reduced space at the level $(\alpha,1) \in \Phi(M) \times \{1\} \subset \partial (\Phi(M) \times [-1,1])$ is not a singleton. 
\end{remark}

Similarly, for each fixed point whose image lies in a short face, the number of its isotropy weights parallel to that face is the same as the dimension of the face.
\begin{lemma}\label{lem:number-of-weights-in-a-short-face}
    Let $(M,\omega,\Phi)$ be a connected complexity one $T$-space with proper moment map. Let $F \subset \Phi(M)_{\short}$ be a face of the moment image. Let $p \in \Phi^{-1}(F)$ be a fixed point. Then the number of isotropy weights that are parallel to $F$ is the same as $\dim(F)$.
\end{lemma}
\begin{proof}
    Let $Y \cong \C^{n+1}$ be the local model associated to $p$ with moment map $\Phi_Y(z) = \Phi(p)+\frac{1}{2}\sum_{j=0}^n|z_j|^2\eta_j$. By Lemma~\ref{lem:local-cone}, there exists a neighborhood $U$ of $\Phi(p)$ such that $\Phi(M) \cap U = (\Phi(p)+ \sum_{j=0}^n \R_{\geq 0} \eta_j) \cap U$. After reordering the weights, we may assume $F \cap U = (\Phi(p) + \sum_{j=0}^k \R_{\geq 0} \eta_i) \cap U$. Since $\eta_0,\ldots,\eta_n$ span the $n$-dimensional vector space $\ft^*$, $k$ is either $\dim(F)$ or $\dim(F) -1$. If $k = \dim (F)$, then $\eta_0,\ldots,\eta_k$ are linearly dependent, so there exist $a_0,\ldots,a_k \in \R$, not all zero, such that $\sum_{j=0}^k a_j \eta_j=0$. Thus, $\sum_{a_j >0} a_j\eta_j = -\sum_{a_j <0} a_j \eta_j$. Choose $z \in \C^{n+1}$ such that $z_j = \sqrt{a_j}$ for all $j$ such that $a_j >0$ and $z_j=0$ otherwise. Choose $w \in \C^{n+1}$ such that $w_j = \sqrt{-a_j}$ for all $j$ such that $a_j <0$ and $w_j=0$ otherwise. Then for $\epsilon>0$ sufficiently small, $\Phi_Y(\epsilon z) = \Phi_Y(\epsilon w) \in F$ but $z,w$ are not in the same $T$-orbit, which is a contradiction to the assumption that $F \subset \Phi(M)_{\short}$. Hence, $k=\dim(F)-1$.
\end{proof}

The lemma above and ~\cite[Remark 2.5 and Lemma 2.6]{KT20} together imply the following.

\begin{corollary}\label{cor:preimage-of-a-face}
    Let $(M,\omega,\Phi)$ be a connected complexity one $T$-space with proper moment map. Let $F \subset \Phi(M)$ be a face of the moment image. Let $H \subset T$ be the connected subgroup such that the affine span of $F$ is a translation of $\fh^\circ$. If $F \subset \Phi(M)_{\short}$, then $\Phi^{-1}(F)$ is a connected symplectic toric $(T/H)$-manifold. Otherwise, $\Phi^{-1}(F)$ is a connected complexity one $(T/H)$-space. 
\end{corollary}

We end this section with a quick review of Duistermaat-Hackman measure~\cite{DH82,DH83} and the wall-crossing formula~\cite{GLS}. We will need these results in the proof of Theorem~\ref{thm:n-1-skeleton} in a later section.

Let $(M,\omega,\Phi)$ be a $2n$–dimensional Hamiltonian $T$–space. The Liouville measure on $M$ is given by integrating the volume form $\frac{\omega^n}{(2\pi)^nn!}$ with respect to the symplectic orientation. The {\bf Duistermaat–Heckman measure} is the pushforward of the Liouville measure by the moment map. Specifically, if $U \subset \ft^*$ is a Borel set, its Duistermaat-Heckman measure is given by 
\[m(U) = \int_{\Phi^{-1}(U)} \frac{\omega^n}{(2\pi)^nn!}.\]

\begin{theorem}[\cite{DH82}]
    Let $(M,\omega,\Phi)$ be an effective Hamiltonian $T$-space. There exists a function $f: \ft^* \to \R$ such that
    \begin{itemize}
        \item $f$ is a polynomial of degree at most $\frac{1}{2}\dim (M) - \dim (T)$ on each connected component of regular values of $\Phi$, and
        \item for any Borel set $U \subset \ft^*$, $m(U) = \int_U f d\lambda$, where $\lambda$ is the Lebesgue measure on $\ft^*$.
    \end{itemize}
\end{theorem}

The function $f$ is called the {\bf Duistermaat-Heckman polynomial}. Duistermaat-Heckman polynomials can be computed combinatorially by the ABBV localization formula in equivariant cohomology (see~\cite{AB84,BV82}). We can also compute these polynomials by the wall-crossing formula introduced by Guillemin-Lerman-Sternberg~\cite{GLS}. We recall a special case that we will use later and refer the readers to Section 3 in~\cite{GLS} for the general case. 

Let $\De_+,\De_-$ be two adjacent connected components of regular values of $\Phi$ such that $\De_+,\De_-$ are separated by a wall\footnote{A wall is a facet of both $\overline{\De}_+$ and  $\overline{\De}_-$. The closure of each connected component of regular values of $\Phi$ inside $\Phi(M)$ is a convex polyhedral set by~\cite[Example 3.7]{GLS}. See Theorem~\ref{thm:convexity-regular-values} for a complete proof.} $W \subset \Phi(M)$ oriented so that the normal vector $\xi \in \ft$ to $W$ is pointing out of $\De_-$ and into $\De_+$. Let $L_\xi: \ft^* \to \R$ denote the linear functional $\langle \cdot,\xi \rangle$. Notice that $L_\xi$ is constant on $W$. Assume that $X:=\Phi^{-1}(W)$ is connected and $\dim X=2 \dim W$. Then $T/S^1_\xi$ acts effectively on $X$, where $S^1_\xi$ is the sub-circle whose Lie algebra is $\R \xi$. For any $T$-fixed point $p \in X$, let $\alpha_1^p,\ldots,\alpha_n^p$ denote the isotropy weights at $p$. After possibly reordering these weights, we can assume that $\langle \alpha_i^p,\xi\rangle =0$ for $i \geq m+1$, where $m = n -\frac{1}{2} \dim(X)$. Moreover, since the $\langle \alpha_i^p,\xi\rangle$'s are the
weights of the isotropy representation of $S^1_\xi$ on the normal bundle of $X$ and $X$ is connected, $\langle \alpha_i^p,\xi\rangle = \langle \alpha_i^q,\xi\rangle$ for any two $T$-fixed points $p,q \in X$. Hence, we will call $\langle \alpha_i,\xi \rangle:= \langle\alpha_i^p,\xi \rangle$ the $S^1_\xi$-isotropy weights of $X$. Finally, we define a constant which normalizes the push-forward of the Liouville measure on $X$. Let $\alpha \in \ft^*$ be any vector such that $L_\xi(\alpha) =1$. Define 
\[c^p_\alpha: = |\det(\alpha^p_{m+1},\ldots,\alpha^p_n,\alpha)|.\]

\begin{lemma}
    Let $p,q \in X$ be $T$-fixed points and let $\alpha,\alpha' \in \ft^*$ such that $L_\xi(\alpha) =L_\xi(\alpha') =1$. Then $c^p_{\alpha} = c^q_{\alpha'}$.
\end{lemma}
    
\begin{proof}
    Since $n-m = \frac{1}{2}\dim X = \dim W = \dim T-1$ and since $T/S^1_\xi$ acts effectively on $X$, the isotropy weights $\alpha_{m+1}^p,\ldots,\alpha_n^p$ form a basis of the primitive lattice $\ell^* \cap \ker(L_\xi)$. Hence, $\alpha' = \alpha + \sum_{i=m+1}^n x_i \alpha_i^p$ for some $x_{m+1},\ldots,x_n \in \R$ and it follows that $c^p_{\alpha} = c^p_{\alpha'}$. Similarly, the isotropy weights $\alpha_{m+1}^q,\ldots,\alpha_n^q$ form a basis of $\ell^* \cap \ker(L_\xi)$. Hence, there exists a matrix $A \in \GL(\ell^*)$ such that $A\alpha_i^p = \alpha_i^q$ and $A\alpha' = \alpha'$. Therefore, $c^p_{\alpha} = c^q_{\alpha'}$.
\end{proof}

We define $c_{\alpha}: =c^p_{\alpha}$.  Now, we are ready to state the wall-crossing formula. See~\cite[Theorem 3.2.10, Equation(3.90)]{GLS} for detailed discussion.
\begin{theorem}[\cite{GLS}]\label{thm:wall-crossing}
    Let $(M,\omega,\Phi)$ be a $2n$-dimensional Hamiltonian $T$-space with proper moment map. Let $f$ be the Duistermaat-Heckman polynomial. Let $\De_+,\De_-$ be two adjacent connected components of regular values of $\Phi$ such that $\De_+,\De_-$ are separated by a wall $W \subset \Phi(M)$ oriented so that the normal vector $\xi \in \ft$ to $W$ is pointing out of $\De_-$ and into $\De_+$. Let $f_\pm$ be polynomials such that $f|_{\De_\pm} =f_{\pm}|_{\De_\pm}$. Let $L_\xi:\ft^* \to \R$ denote the linear functional $\langle \cdot, \xi \rangle$. Fix $\alpha \in \ft^*$ such that $L_\xi(\alpha) =1$. Assume that $\Phi^{-1}(W)$ is connected of dimension $2 \dim W$. Let $p \in \Phi^{-1}(W)$ be a $T$-fixed point with isotropy weights $\alpha_1,\ldots,\alpha_n \in \ft^*$ such that $\langle \alpha_i, \xi \rangle =0$ for $i \geq m+1$, where $m = n - \dim (W)$. Then,
    \[f_+ - f_- = \frac{1}{(m-1)!} \frac{1}{\prod_{i=1}^m L_\xi(\alpha_i)}\frac{(L_\xi-L_\xi(W))^{m-1}}{|\det(\alpha_{m+1},\ldots,\alpha_n,\alpha)|}.\]
\end{theorem}

\begin{remark}
    For any positive number $k \in \R$, replacing $\xi$ by $k\xi$ yields the same expression on the right hand side (notice that $\alpha$ needs to be replaced by $\frac{1}{k}\alpha$ correspondingly), so once we choose a normal vector that gives the correct orientation, the exact choice of the normal vector does not matter in the formula.
\end{remark}

\section{Convexity}\label{sec:convexity}
In this section, we prove the following convexity result for Hamiltonian $T$-spaces with proper moment map and discuss how this result applies in various settings.

\begin{proposition}\label{prop:convexity-abstract}
Let  $(M,\omega,\Phi)$ be a connected Hamiltonian $T$-space with  proper moment map.
Let $A$ be a closed subset of $M$ so that for all $p \in A$ there exist  neighborhoods $V_p$ of $p$
and  $U_p$ of $\Phi(p)$ such that $$\Phi(A \cap V_p) \cap \Phi(M)^\circ \cap U_p= \Phi_p(S_p) \cap \Phi_p(Y_p)^\circ \cap U_p,$$ 
where $S_p \subseteq Y_p$ denotes the set of singular points  of the moment map $\Phi_p \colon Y_p \to \mathfrak t^*$
  for the local model $Y_p$ associated to $p$.  
  Let $\mathcal{R} := \Phi(M)^\circ \setminus {\Phi}(A)$.
  Then the following holds:

\begin{enumerate} 
 \item  Given $\alpha \in \Phi(M)$, there is a neighborhood $U_\alpha$ of $\alpha \in \ft^*$ such that
\[\mathcal R \cap U_\alpha = \underset{ p \in \Phi^{-1}(\alpha) \cap  A }{\medcap} \mathcal R_p \cap U_\alpha,\footnote{Here, we follow the convention that if $\Phi^{-1}(\alpha) \cap A = \emptyset$, then the right hand side is $\ft^* \cap U_\alpha = U_\alpha$.}\] where $\mathcal{R}_p = \Phi_p(Y_p)^\circ \setminus \Phi_p(S_p)$ is the set of regular values of $\Phi_p$ inside $\Phi_p(Y_p)$.
\item 
Each connected component  $\mathcal{C}$ of $\mathcal{R}$ is the interior of a convex rational locally polyhedral set $\overline{\mathcal C}$; moreover,
$\Phi(M)$ is the union of all such $\overline{\mathcal C}$.
\end{enumerate}
\end{proposition}

To prove Proposition~\ref{prop:convexity-abstract}, we will prove a local convexity result, and then use
a result of Cel~\cite{Cel} to show that this implies
global convexity.

Let $B$ be a nonempty subset of $\R^n$. 
Following Cel~\cite{Cel}, we say that a point $\alpha \in \overline{B}$ is a \textbf{point of strong local C-convexity of $B$} if and only if there exists a neighborhood $U$ of $\alpha$ in $\R^n$ such that each component of $B \cap U$ is convex.
The following is a special case of the main 
result of~\cite{Cel}.

\begin{theorem}[Cel \cite{Cel}] \label{thm:cel-convexity}
Let $B$ be an open connected subset of $\R^n$.
If all points of the boundary
of $B$ are points of strong 
local C-convexity, then $B$ is convex.
\end{theorem}

To prove local convexity we will need the following variant of Caratheodory's theorem for cones.

\begin{lemma}\label{lem:span-one-less}
Given  $\eta_1,\ldots, \eta_{\ell}\in\R^k$  and  $\beta \in \sum_{i=1}^\ell \R_{\geq 0} \eta_i$, there exists
 $I \subseteq \{1,\dots,\ell\}$ such that $\{\eta_i\}_{i \in I}$ is linearly independent and $\beta \in \sum_{i \in I} \R_{\geq 0} \eta_i$.
\end{lemma}

\begin{proof}

Assume that
$\{\eta_1, \ldots, \eta_\ell\}$ are  linearly dependent; we will show that  $\beta$ is in the non-negative span of a proper subset of $\{\eta_1, \ldots, \eta_\ell\}$. Repeating this argument, if necessary, $\beta$ lies in
the non-negative span of a linearly independent subset.

By assumption, there exists $x_1,\ldots,x_{\ell}\geq 0$ such that  $\beta = \sum_{i=1}^{\ell} x_i \eta_i.$
 If $x_i = 0$ for some $i$, then  $\beta$ is in the non-negative span of the remaining vectors, so we may  now assume that $x_1,\ldots, x_{\ell}>0$.
 By assumption,  there exist $y_1,\ldots,y_{\ell}\in\R$, not all zero,  such that $\sum_{i=1}^{\ell} y_i \eta_i = 0$.
 We may assume that at least one of the $y_i$ satisfies $y_i<0$, otherwise replace each $y_i$ with $-y_i$. Thus, for any $a \in \R$ we have that
 \[
  \beta = \sum_{i=1}^{\ell}x_i\eta_i + a \left( \sum_{i=1}^{\ell} y_i \eta_i \right) = \sum_{i=1}^{\ell}(x_i+a  y_i)\eta_i.
 \]
 Let $a_0>0$ be the smallest positive value such that $x_{i_0}+a_0 y_{i_0}=0$ for some $i_0$, such an $a_0$ must exist since all $x_i$ are positive and at least one of the $y_i$ is negative.
 Then $\beta$ is in the non-negative span of $\{\eta_1,\ldots,\eta_{\ell}\}\setminus\{\eta_{i_0}\}$. 
\end{proof}

Next, we describe the connected components of the regular values of the moment map for a local model.

\begin{lemma}\label{lem:rvlnf}
Let $H \subseteq T$ act on $\C^{n}$ with weights $\eta_1,\dots,\eta_n \in \fh^*$.
Let $Y = T \times_H \fh^\circ \times \C^{n}$ be the associated local model and $\Phi_Y: Y \to \ft^*$ be the homogeneous moment map. 
If  $\mathcal{C}\subseteq \Phi_Y(Y)$ is a connected component of regular values of $\Phi_Y$, then for any $\beta \in \mathcal C$,
\begin{equation*}
\mathcal{C} = \fh^\circ + \medcap \big \{\sum_{i \in I} \R_{> 0} \eta_i \ \big| \  I \subseteq \{1,\dots,n\} \text{ and }   \beta \in \fh^\circ + \sum_{i \in I} \R_{> 0} \eta_i . \big \}.
\end{equation*}\end{lemma}

\begin{proof}
The homogeneous moment map is given by $\Phi_Y([t,\nu,z]) = \nu + \Phi_H(z)$, where $\Phi_H \colon \C^n \to \fh^*$
is the homogeneous moment map for the $H$ action on $\C^n$.
Hence, it is enough to prove that if $\mathcal{C}_H \subseteq \Phi_H(\C^n)$ is a connected component of regular values of $\Phi_H$, then for any $\beta \in \mathcal C_H$
\begin{equation}\label{eqn:chambers-intersection}
\mathcal{C}_H = \medcap \big\{\sum_{i \in I} \R_{> 0} \eta_i  \ \big| \  I \subseteq \{1,\dots,n\} \text{ and }   \beta \in \sum_{i \in I} \R_{> 0} \eta_i  \big \}.
\end{equation}

The homogeneous moment map is  $\Phi_H(z_1,...,z_n) = \frac{1}{2}\sum_{i=1}^n |z_i|^2 \, \eta_i$.  Hence, by Lemma~\ref{lem:span-one-less}, every point $\alpha$ in the
moment image $\Phi_H(\C^n)$ can be written as a (strictly) positive linear combination 
of linearly independent  weights; moreover, $\alpha$  is a singular value of $\Phi_H$
 exactly if it can be written as a positive linear combination of fewer than $h := \dim H$
 linearly independent weights.    
 In particular, since $\beta$ is regular value, any linearly independent set of weights which includes $\beta$ in its positive span is a basis. 

      Let  $\mathcal{R}$ be the set of points in $\Phi_H(\C^n)$ that are regular values of $\Phi_H$, and let $\mathcal{P}$ denote the set in the right hand side of Equation \eqref{eqn:chambers-intersection}.
    
    Whenever $\{ \eta_i \colon i \in I\}$ is a basis, the above argument implies that $(\sum_{i \in I} \R_{> 0} \eta_i) \cap \mathcal{R} = (\sum_{i \in I} \R_{\geq 0} \eta_i ) \cap \mathcal{R}$, so $(\sum_{i \in I} \R_{> 0} \eta_i) \cap \mathcal{R}$ is both open and closed in $\mathcal{R}$. Hence, $\mathcal P \cap \mathcal{R}$ is a  union of connected components of regular values. 
    Since $\mathcal{P}$ is convex and thus connected, to complete the proof it is sufficient to show that $\mathcal{P}\subseteq\mathcal{R}$.

 So assume $\mathcal{P} \not\subseteq \mathcal{R}$.  Since $\mathcal{R}$ is open and $\mathcal{P}$ is convex, there exists a singular value  $\gamma\in \mathcal{P}$  such that $\beta_t = (1-t)\beta+t\gamma$ is regular for $t\in[0,1)$.
    Since $\gamma$ is singular, $\gamma\in \sum_{i \in J} \R_{> 0} \eta_i$ for some index set $J \subset \{1,\ldots,n\}$ with $j := |J|<h$ such that $\{\eta_i:i \in J\}$ is a set of linearly independent weights. 
    
    Let $\pi \colon \fh^* \to\R^{h-j}$ be a  projection with $\mathrm{ker}(\pi) = \sum_{i \in J} \R \eta_i$. 
    Since $\beta \in \Phi_H(\C^n) = \sum_{i=1}^n \R_{\geq 0} \eta_i$,  $\pi(\beta) \in \sum_{i \notin J} \R_{\geq 0} \pi(\eta_i)$. By Lemma~\ref{lem:span-one-less}, there is a set of linearly independent weights $\{\pi(\eta_i): i \in J'\}$ in $\R^{h-j}$ such that $\pi(\beta)\in \sum_{i \in J'} \R_{> 0} \pi(\eta_i)$, so $\beta \in \sum_{i \in J'}\R_{> 0} \eta_i + \sum_{i \in J} \R \eta_i$. 
    Since $\gamma\in \sum_{i\in J}\R_{>0}\eta_i$, for $t<1$ sufficiently close to 1, $\beta_t \in \sum_{i\in J\cup J'} \R_{>0}\eta_i$. Since $\beta_t$ is regular for all $t\in [0,1)$, we conclude that $\{\eta_i: i \in J \cup J'\}$ is a basis. Moreover, $\beta \in \sum_{i\in J\cup J'}\R_{>0}\eta_i$, since otherwise there would be some $\beta_{t_0}$ for which the coefficient of one of the $\{\eta_i\mid i\in J\}$ is zero, and therefore would be singular.
    On the other hand, since $\gamma\in \sum_{i \in J} \R_{> 0} \eta_i$, then $\gamma\notin \sum_{i \in J\cup J'} \R_{> 0} \eta_i$, so $\gamma \not \in \mathcal P$,  which is a contradiction.
\end{proof}

Now we are ready to prove Proposition~\ref{prop:convexity-abstract}.

\begin{proof}[Proof of Proposition~\ref{prop:convexity-abstract}]
The  moment map $\Phi$ is closed because it's proper and because $\ft^*$ locally compact and Hausdorff.  
Hence, $\De :=\Phi(M)$ and $\Phi(A)$  are closed.
Thus, the set $\mathcal{R}$ is open in $\ft^*$.

(1) Fix  $\alpha\in \Phi(M)$ and $p\in\Phi^{-1}(\alpha)$.
If $p\notin A$, then since $A$ is closed, there exists a neighborhood $V_p$ such that $V_p\cap A = \emptyset$. Combining this with the assumption on $A$, for all $p$, there
exist neighborhoods $V_p$ of $p$
and $U_p$ of $\Phi(p)$ such that 
$$\Phi(A \cap V_p) \cap \De^\circ \cap U_p = \begin{cases} \emptyset & \text{if } p \not\in A, \\
\Phi_p( S_p) \cap \Phi_p(Y_p)^\circ \cap U_p & \text{if } p \in A. \end{cases}$$
Since $\Phi^{-1}(\alpha)$ is compact, there exist $p_1,...,p_\ell \in \Phi^{-1}(\alpha)$ 
such that $\Phi^{-1}(\alpha) \subset \medcup_{i=1}^\ell  V_{p_i}$. Since there are only a finite number of different slice representations for points in $\Phi^{-1}(\alpha)$,
we may further assume that  $p_i$ is in $A$ exactly if $i \leq k$, and 
that every slice representation of a point $p \in \Phi^{-1}(\alpha) \cap A$ 
is represented by $p_i$ for some $1 \leq i \leq k$. 
Hence,
$$\mathcal{R}_\alpha:=\underset{ p \in \Phi^{-1}(\alpha) \cap  A }{\medcap} \mathcal R_p  = \medcap_{i=1}^k \mathcal R_{p_i}. $$
By Lemma~\ref{lem:rvlnf}, there are a finite number of connected
components of $\mathcal R_p$, each of which is the interior of
a convex rational polyhdeal cone with vertex $\alpha$.  
Therefore, there are a finite number of connected components of   $\mathcal R_\alpha$, each of
which is the interior of a convex rational polyhedral cone with vertex $\alpha$.

Since the  moment map $\Phi$ is closed and $\Phi^{-1}(\alpha) \subset \medcup_{i=1}^\ell  V_{p_i}$, there exists an open neighborhood $U_0$ of $\alpha$ such that $\Phi^{-1}(U_0) \subseteq \medcup_{i=1}^\ell V_{p_i}$. 
Let $U_\alpha \subseteq \bigcap_{i=0}^\ell U_{p_i}$ be a convex neighborhood of $\alpha$. By Lemma~\ref{lem:local-cone}, we may assume that $\De^\circ \cap U_\alpha = \Phi_{p_i}(Y_{p_i})^\circ \cap U_\alpha$ for all $i =1,\ldots,\ell$.
Then
$$ \big(\De^\circ \setminus \Phi(A \cap V_{p_i}) \big) \cap U_\alpha = \begin{cases}   \big( \Phi_{p_i}(Y_{p_i})^\circ \setminus \Phi_{p_i}(S_{p_
i}) \big) \cap U_\alpha & \text{if } 1 \leq i \leq k ,\\
\De^\circ \cap U_\alpha & \text{if } k < i \leq \ell. \end{cases}$$
Hence,
\begin{multline*}
\mathcal R \cap U_\alpha = 
(\De^\circ \setminus {\Phi}(A)) \cap U_\alpha  = \big(\De^\circ \setminus \Phi\big( A \cap \big(\medcup_{i=1}^\ell V_{p_i}\big)\big) \big)\cap U_\alpha 
=\medcap_{i=1}^\ell \big(\De^\circ  \setminus  {\Phi}(A  \cap V_{p_i}) \big)\cap U_\alpha \\ 
= \medcap_{i=1}^k  \big(\Phi_{p_i}(Y_{p_i})^\circ \setminus \big( \Phi_{p_i} (S_{p_i}) \big) \cap U_\alpha 
= \medcap_{i=1}^k \mathcal{R}_{p_i} \cap U_\alpha = \underset{ p \in \Phi^{-1}(\alpha) \cap  A }{\medcap} \mathcal R_p \cap U_\alpha.
\end{multline*}

(2) Now let $\mathcal C$ be a component of $\mathcal R$
such that $\alpha \in \overline{\mathcal C} \subseteq \Phi(M).$ Since $U_\alpha$ is a convex neighborhood of $\alpha$, $\mathcal R \cap U_\alpha = \mathcal R_\alpha \cap U_\alpha$, and every component of $\mathcal R_\alpha$ is convex,   each component of $\mathcal C \cap U_\alpha$ is convex.  
 Since $\mathcal{C}$ is open, by Theorem~\ref{thm:cel-convexity} this implies that $\mathcal{C}$, and hence $\overline{\mathcal{C}}$, is convex.
Since  each component of $\mathcal R_\alpha$ is
the interior of a convex rational polyhedral 
cone with vertex $\alpha$,  
$\overline{\mathcal C}$ is a rational locally polyhedral set. 
Similarly, since $\mathcal R_\alpha$ has a finite number of components, $\Phi(M)$ is the union of all such $\overline{\mathcal C}$.
(2) follows immediately.
\end{proof}

Before applying Proposition~\ref{prop:convexity-abstract} to different settings, we first prove that the restriction of the homogeneous moment map of a local model to the singular values is open at $0$.

\begin{lemma}\label{lem:open-at-0}
Let $Y = T \times_H \fh^\circ \times \C^{n}$ and let $S_Y \subset Y$  be the set of singular points of the homogeneous  moment map $\Phi_Y \colon Y \to \ft^*$.
Given an open neighborhood $V$ of $[1,0,0]$ in $Y$, there exists an open neighborhood $U$ of $0$ in  $\mathfrak{t}^*$  such that $\Phi_Y(S_Y) \cap U \subseteq \Phi_Y(S_Y \cap V).$
\end{lemma}

\begin{proof}

Let $\eta_1,\ldots,\eta_n \in \fh^*$ denote the weights of the $H$-action on $\C^n$. Given $I \subseteq \{1,\dots,n\}$,
let $\C^I = \{ z \in \C^n \mid z_j = 0 \ \forall \  j \not \in I \}.$
Since  $\Phi_Y([t,\nu,z] ) = \frac{1}{2} \sum_i \eta_i |z_i|^2 + \nu,$
the moment map is singular at  $[t,\nu,z] \in Y$  exactly if $z \in \C^I$
  for some $I$ such that $\{\eta_i\}_{i \in I}$ does not span $\mathfrak{h}^*$.
Hence, by Lemma~\ref{lem:span-one-less},   $\Phi_Y(S_Y) = \cup_{I \in \mathcal I} \Phi_Y(T \times_H \fh^\circ \times \C^I)$, where $\mathcal I$ is
 the set of $I \subseteq \{1,\dots,n\}$ such that the weights $\{\eta_i\}_{i \in I}$
 are linearly independent but do not span $\fh^*$.
Moreover,   $\Phi_Y$ restricts to an open map from $T \times_H \fh^\circ \times \C^I$ to $\Phi_Y(T \times_H \fh^\circ \times \C^I)$ for all $I \in \mathcal I$;
since the weights $\{\eta_i\}_{i \in I}$ are linearly independent,
 the restriction can be written as the composition of two open maps: the quotient map from $T \times_H \fh^\circ \times \C^I$ to $\fh^\circ \times \R_{\geq 0}^I$,
 and an invertible linear transformation from $\fh^\circ \times \R^I$ to $\fh^\circ + \sum_{i \in I} \R \eta_i$.
In particular, there exists a neighborhood $U_I$ of $0$ in $\ft^*$ such that $\Phi_Y\left(T \times_H \fh^\circ \times \C^I\right) \cap U_I \subseteq \Phi_Y \left( \left(T \times_H \fh^\circ \times \C^I  \right) \cap V \right).$
Define $U := \cap_{I \in \mathcal I} U_I$. Then
 $\Phi_Y(S_Y) \cap U = \cup_{I \in \mathcal I} \Phi_Y\left(T \times_H \fh^\circ \times \C^I\right) \cap U  \subseteq \Phi_Y(S_Y \cap V).$
\end{proof}

We use Proposition~\ref{prop:convexity-abstract} to derive the next result, which describes the connected components of regular values of the moment map. Theorem~\ref{thm:convexity-of-chambers} follows immediately from the following result. Item (2) in the compact case is well-known (cf.~\cite[Example 3.7]{GLS}). 

\begin{theorem}\label{thm:convexity-regular-values}
Let  $(M,\omega,\Phi)$ be a connected Hamiltonian $T$-space with proper moment map $\Phi: M \to \mathfrak{t}^*$. Let $\mathcal{R}$ denote the set of points in $\Phi(M)$ that  are regular values of $\Phi$. Then the following holds:
\begin{enumerate} 
  \item Given $\alpha \in \Phi(M)$, there is a neighborhood $U_\alpha$ of $\alpha \in \ft^*$ such that
\[\mathcal R \cap U_\alpha = \underset{ p \in \Phi^{-1}(\alpha)}{\medcap} \mathcal R_p \cap U_\alpha,\] where $\mathcal{R}_p$ denotes the set of points in $\Phi_p(Y_p)$ that are regular values of the moment map $\Phi_p \colon Y_p \to \mathfrak t^*$
  for the local model $Y_p$ associated to $p$. 
  \item Each connected component $\mathcal{C}$ of $\mathcal{R}$  is the interior of a convex rational locally polyhedral set $\overline{\mathcal{C}}$; moreover, $\Phi(M)$ is the union of all such $\overline{\mathcal C}$.
\end{enumerate}
\end{theorem}

\begin{proof}
The set $S \subset M$ of  singular points of $\Phi$ is closed. Given $p \in M$,
let $S_p \subseteq Y_p$ denote the set of singular points  of the moment map $\Phi_p \colon Y_p \to \mathfrak t^*$
for the local model $Y_p$ associated to $p$. 
If $p \in S$, then by the local normal form theorem and Lemma~\ref{lem:open-at-0},  there exist open neighborhoods $V_p \subseteq M$ of $p$ and $U_p \subseteq \ft^*$ of $\Phi(p)$ such that $\Phi(S\cap V_p)\cap U_p = \Phi_p(S_p)\cap U_p$.
By Lemma~\ref{lem:local-cone},   $\Phi_p(Y_p) \cap U_p = \Phi(M) \cap U_p$  after possibly further shrinking $U_p$, and hence $$\Phi(S \cap V_p) \cap \Phi(M)^\circ \cap U_p = \Phi_p(S_p) \cap \Phi_p(Y_p)^\circ \cap U_p.$$ 
Fix $\alpha \in \Phi(M)$ and  $p \in  \Phi^{-1}(\alpha)$.  Since  the boundary of $\Phi_p(Y_p)$ consists of singular values of $\Phi_p$,  the set of points in $\Phi_p(Y_p)$ that are regular values of $\Phi_p$ is $\Phi_p(Y_p)^\circ \setminus \Phi_p(S_p)$.  Moreover, if $p \not\in S$ then $S_p = \emptyset$. 
Therefore, we may assume that $\underset{ p \in \Phi^{-1}(\alpha) \cap S}{\medcap} \mathcal R_p \cap U = \underset{ p \in \Phi^{-1}(\alpha)}{\medcap} \mathcal R_p \cap U$ for some neighborhood $U$ of $\alpha$.
Finally, since the boundary of $\Phi(M)$ consists of  singular values of $\Phi$,
    $ \mathcal{R} = \Phi(M)^\circ \setminus \Phi(S)$. 
Therefore, the claim follows immediately by applying Proposition~\ref{prop:convexity-abstract} to $S\subset M$.
\end{proof}

The following corollary of Theorem~\ref{thm:convexity-regular-values} will be used in proving Theorem~\ref{thm:skeleton-connected} in the last section. Recall that given a face $F$ of a convex set $P$, the {\bf relative interior} of $F$, denoted by $\relint(F)$, is the interior of $F$ inside the affine span of $F$ under the subspace topology.

\begin{corollary}\label{cor:short-facet}
     Let $(M,\omega,\Phi)$ be a complexity one $T$-space with proper moment map $\Phi$. Let $F_1$ be a facet of $\Phi(M)$ such that $F_1 \subseteq \Phi(M)_{\short}$. Fix a connected component $\mathcal{C}$ of the set of regular values $\mathcal{R}$ inside $\Phi(M)$ and let $F_2$ be a facet of $\overline{\mathcal{C}}$. If $F_2 \subseteq F_1$, then $F_2=F_1$.
\end{corollary}
\begin{proof}
     Fix $\alpha \in \relint(F_1)$, and take any point $p \in \Phi^{-1}(\alpha)$. Let $Y= T \times_H \fh^\circ \times \C^{h+1}$ be the local model associated to $p$ with moment map $\Phi_Y([t,\nu,z]) = \alpha+ \frac{1}{2}\sum_{j=0}^h \eta_j |z_j|^2 + \nu$, where $\eta_0,\ldots,\eta_h$ are isotropy weights at $p$. Since $F_1$ is a facet of $\Phi(M)$, Lemma~\ref{lem:local-cone} implies that $h=1$. Since $\alpha \in \Phi(M)_{\short}$, Lemma~\ref{lem:tall/exc-criteria} implies that $\eta_0,\eta_1$ are positive multiples of each other. Hence, the set of regular values $\mathcal{R}_Y$ of $\Phi_Y$ inside $\Phi_Y(Y)$ can be written as \[\mathcal{R}_Y = \alpha+ \R_{>0} \eta_0 + \fh^\circ.\]
     
     Since $\alpha \in \Phi(M)_{\short}$, item (1) of Theorem~\ref{thm:convexity-regular-values} implies that there exists an open neighborhood $U$ of $\alpha$ such that $\mathcal{R} \cap U = \mathcal{R}_Y \cap U$. Since $\mathcal{R}_Y$ is connected, there exists a unique connected component $\mathcal{C}_\alpha$ of $\mathcal{R}$ such that $\mathcal{R} \cap U = \mathcal{C}_\alpha \cap U$ and $\alpha \in \overline{\mathcal{C}}_\alpha$, so there exists a well-defined function assigning each $\alpha \in \relint(F_1)$ to a connected component $\mathcal{C}_\alpha$ of $\mathcal{R}$. Moreover, this function is locally constant. Since $\relint(F_1)$ is connected, this function must be constant. Since $F_2 \subseteq F_1$ are both polyhedral sets of the same dimension, there exists $\beta \in F_2 \cap \relint(F_1)$, so for sufficiently small neighborhood $U_\beta$ of $\beta$, $\mathcal{C} \cap U_\beta$ is nonempty and $\mathcal{R} \cap U_\beta = \mathcal{C}_\beta \cap U_\beta$. Hence, $\mathcal{C} ,\mathcal{C}_\beta$ are the same connected components. Therefore, for any $\alpha \in \relint(F_1)$, $\alpha \in\overline{\mathcal{C}}_\alpha=\overline{\mathcal{C}}$ and thus $F_1 \subset \overline{\mathcal{C}}$.  Let $\mathcal{H}$ be the supporting hyperplane through $F_1$ and $F_2$. Then $F_1 \subseteq \mathcal{H} \cap \overline{\mathcal{C}} =F_2$.
\end{proof}

We end this section by discussing the special case of Lemma~\ref{lem:rvlnf} in the complexity one setting.

\begin{corollary}\label{cor:cx-1-rvlnf}
Let $Y = T \times_H \fh^\circ \times \C^{h+1}$ be a complexity one local model with moment map $\Phi_Y([t,\nu,z]) =\frac{1}{2} \sum_{j=0}^{h} \eta_j |z_j|^2 + \nu$. Let $\xi_0,\ldots,\xi_h$ be the exponents of the defining monomial of $Y$. Let $\mathcal{C}$ be a connected component of regular value of $\Phi_Y$ inside $\Phi_Y(Y).$
\begin{enumerate}
    \item If $Y$ is tall, then there exists a unique $i \in \{0,\ldots,h\}$ such that $\mathcal{C}= \fh^\circ+\sum_{j \neq i} \R_{> 0} \eta_j$. Moreover, $\xi_{i} \neq 0$.
    \item If $Y$ is not tall, then there exist exactly two distinct indices $i,i' \in \{0,\ldots,h\}$ such that $\mathcal{C}= \fh^\circ+(\sum_{j \neq i} \R_{> 0} \eta_j \cap \sum_{j \neq i'} \R_{> 0} \eta_j)$. Moreover, $\xi_{i}\xi_{i'}<0$.
\end{enumerate}
\end{corollary}

\begin{proof}
    Let $\beta \in \mathcal{C}$ be a regular value. Since $\beta \in \fh^\circ +\sum_{j=0}^h \R_{\geq 0}\eta_j$, Lemma~\ref{lem:span-one-less} implies that there exists $i \in \{0,\ldots,h\}$ such that $\{\eta_j: j\neq i\}$ is a basis of $\fh^*$ and $\beta= \nu+\sum_{j \neq i} x_j \eta_j$ for some $\nu \in \fh^\circ$ and $x_j >0$. Moreover, if $\xi_i =0$, then $\sum_{j \neq i} \xi_j \eta_j =0$, which contradicts the fact that $\{\eta_j: j \neq i\}$ is a basis of $\fh^*$. Hence, $\xi_i \neq 0$.     
    If there exists a different $i'\in \{0,\ldots,h\}$ such that $\beta \in \fh^\circ+\sum_{j \neq i'} \R_{> 0} \eta_j$, then there exists $a_j \in \R$ such that $\sum_{j=0}^h a_j \eta_j=0$ with $a_i >0$ and $a_{i'}<0$. Hence, the last claim of Lemma~\ref{lem:defining-monomial} and Lemma~\ref{lem:rvlnf} together imply that when $Y$ is tall, $\mathcal{C}= \fh^\circ+\sum_{j \neq i} \R_{> 0} \eta_j$ for a unique $i \in \{0,\ldots,h\}$.

    When $Y$ is not tall, notice that there exists a unique $t_0 \in \R$ such that $t_0\xi_i >0$, $x_j+t_0\xi_j\geq0$ for all $j \neq i$, and $x_{i'}+t_0\xi_{i'}=0$ for at least one $i'$. Since $\beta$ is a regular value, such $i'$ must be unique, so $\beta \in \fh^\circ+\sum_{j \neq i'} \R_{> 0} \eta_j$ and $\xi_i\xi_{i'} <0$. Finally, if there is a third index $k \neq i,i'$ such that $\beta \in \fh^\circ+\sum_{j \neq k} \R_{> 0} \eta_j$, then the same argument implies that $\xi_i \xi_k<0$ and $\xi_{i'}\xi_k<0$, which is impossible. Hence, by Lemma~\ref{lem:rvlnf},  $\mathcal{C}= \fh^\circ+(\sum_{j \neq i} \R_{> 0} \eta_j \cap \sum_{j \neq i'} \R_{> 0} \eta_j)$. 
\end{proof}

\section{Connectedness}\label{sec:connectedness}
In this section, we prove Theorem~\ref{thm:n-1-skeleton} which states that the moment image of the set of exceptional orbits with codimension at least $1$ is connected. This result will imply Theorem~\ref{thm:skeleton-connected}. Then, we use Theorem~\ref{thm:skeleton-connected} and results in~\cite{LPT} to prove Theorem~\ref{thm:max-short-bound}.

Let $(M,\omega,\Phi)$ be a complexity one $T$-space. Recall that $(M,\omega,\Phi)$ is {\bf maximally short} if $\Phi(M)_{\short} = \partial \Phi(M)$, where $\Phi(M)_{\short}$ denotes the set of points in $\Phi(M)$ whose reduced space is a singleton. The {\bf m-skeleton} $\Sigma_m$ is the set of tall exceptional orbits that are at most $m$-dimensional. 
We need the following two results to prove Theorem~\ref{thm:n-1-skeleton}.

\begin{lemma}\label{lem:regular-values}
    Let $(M,\omega,\Phi)$ be a connected $(2n+2)$-dimensional complexity one $T$-space with proper moment map and let $\Sigma_{n-1}$ be its $(n-1)$-skeleton. Let $\mathcal{R}$ denote the set of regular values of $\Phi$ inside $\Phi(M)$. Then $\mathcal{R} = \Phi(M)^\circ \setminus \Phi(\Sigma_{n-1})$.
\end{lemma}

\begin{proof}
    A point $p \in M$ is singular if and only if the dimension of its stabilizer group is at least one. Hence, any point whose orbit lies in $\Sigma_{n-1}$ is singular. On the other hand, if $\alpha \in \Phi(M)^\circ$ is a singular value of $\Phi$, then there exists a point $p \in \Phi^{-1}(\alpha)$ whose stabilizer group has dimension at least one. By~\cite[Example 1.4]{KT14}, $p$ is an exceptional point. Moreover, by~\cite[Lemma 5.4]{KT01} any point in $\Phi^{-1}(\Phi(M)^\circ)$ is tall, so the orbit through $p$ is in $\Sigma_{n-1}$ and $\alpha \in \Phi(\Sigma_{n-1})$. We conclude that the set of singular values of $\Phi$ in $\Phi(M)^\circ$ is equal to $\Phi(\Sigma_{n-1})$. Since $\partial \Phi(M)$ consists of singular values of $\Phi$, the claim follows immediately.
\end{proof}

Another key ingredient is a wall-crossing formula of the Duistermaat-Heckman function in the complexity one setting.

\begin{lemma}\label{lem:cx-1-wall-crossing}
    Let $(M,\omega,\Phi)$ be a complexity one $T$-space with proper moment map. Let $f$ be the Duistermaat-Heckman polynomial. Let $\mathcal{C} \subset \Phi(M)$ be a connected component of regular values of $\Phi$ such that a facet $F \subset \Phi(M)_{\short}$ is also a facet of $\overline{\mathcal{C}}$.  Let $\mathcal{H}_{X}(c) = \{\alpha \in \ft^*: \langle \alpha, X \rangle = c\}$ be the supporting hyperplane through $F$ with an inward-pointing normal vector $X \in \ft$, i.e., $\Phi(M) \subset  \mathcal{H}_{X}^+(c):=\{\alpha \in \ft^*: \langle \alpha, X \rangle > c\}$. Let $f_{\mathcal{C}}$ be the polynomial such that $f|_{\mathcal{C}} =f_{\mathcal{C}}|_{\mathcal{C}}$. Then, there exists $A >0$ such that
    \[f_{\mathcal{C}} =  \frac{\langle \cdot, X \rangle - c}{A}.\]
\end{lemma}

\begin{proof}
    By Corollary~\ref{cor:preimage-of-a-face}, since $F \subset \Phi(M)_{\short}$, $\Phi^{-1}(F)$ is a connected manifold of dimension $2\dim(F)$. Let $p \in \Phi^{-1}(F)$ be a $T$-fixed point with isotropy weights $\eta_1,\ldots,\eta_n$. By Lemma~\ref{lem:number-of-weights-in-a-short-face} and since $X$ is an inward-pointing normal, we may assume that $\langle \eta_i, X \rangle =0$ for $i \geq 3$ and $\langle \eta_i, X \rangle  >0$ for $i=1,2$. We apply Theorem~\ref{thm:wall-crossing} by crossing the wall $F$ that separates $\De^- = \ft^* \setminus \Phi(M)$ and $\De^+ = \mathcal{C}$. Since $f$ is supported on $\Phi(M)$, for any $\alpha \in \ft^*$ such that $\langle \alpha,X \rangle =1$,
    \[f_{\mathcal{C}} -0 = \frac{1}{\langle\eta_1,X\rangle \langle\eta_2, X \rangle}\frac{\langle \cdot , X \rangle - c}{|\det(\eta_3,\ldots,\eta_n,\alpha)|}.\]
    The claim follows by setting $ A = \langle\eta_1,X\rangle \langle\eta_2, X \rangle|\det(\eta_3,\ldots,\eta_n,\alpha)|>0$.
\end{proof}

Now, we are ready to prove our main result in this section.
\begin{theorem}\label{thm:n-1-skeleton}
    Let $(M,\omega,\Phi)$ be a compact, connected, $(2n+2)$-dimensional maximally short complexity one $T$-space. Let $\Sigma_{n-1}$ be its $(n-1)$-skeleton. If $n \geq 2$, then $\Phi(\Sigma_{n-1})$ is connected.
\end{theorem}

\begin{proof}
    We prove by contradiction. Suppose not, then $\Phi(\Sigma_{n-1})$ has connected components $S_1,\ldots,S_k$ for some $k \geq 2$. Our first claim is that there exists a connected component $\mathcal{C}$ of $\mathcal{R}$ whose closure intersects two different connected components. Suppose not, then the closure of every connected component of $\mathcal{R}$ intersects at most one of $S_i$. Define \[\mathcal{R}_i:= \underset{{\overline{\mathcal{C}} \cap S_i \neq \emptyset}}{\sqcup} \mathcal{C} \quad \textrm{and} \quad \mathcal{R}_0=\underset{{\partial\mathcal{C} \subset \partial \Phi(M)}}{\sqcup} \mathcal{C}.\]   
    By Theorem~\ref{thm:convexity-regular-values} part (2), $S_i = S_i \cap \Phi(M) = S_i \cap (\cup \overline{\mathcal{C}}) \subset \underset{{\overline{\mathcal{C}} \cap S_i \neq \emptyset}}{\cup} \overline{\mathcal{C}} = \overline{\mathcal{R}}_i$. By Lemma~\ref{lem:regular-values}, $\Phi(M)^\circ = \mathcal{R} \cup (\sqcup_{i=1}^k S_i) =  \sqcup_{i=1}^k (\mathcal{R}_i \cup S_i) \sqcup \mathcal{R}_0$. By assumption, $\mathcal{R}_j \cup S_j$ is disjoint from  $\overline{\mathcal{R}}_i$ for any $i \neq j$, so $\overline{\mathcal{R}}_i \cap \Phi(M)^\circ = \overline{\mathcal{R}}_i  \cap (\mathcal{R}_i \cup S_i) = \mathcal{R}_i \cup S_i$.  
    Hence, $\mathcal{R}_i \cup S_i$ is a closed subset of $\Phi(M)^\circ$ for each $1 \leq i \leq k$. Moreover, by definition $\overline{\mathcal{R}}_0 \cap \Phi(\Sigma_{n-1}) = \emptyset$, so $\mathcal{R}_0 = \overline{\mathcal{R}}_0 \cap \Phi(M)^\circ$ is also closed. Thus, $\Phi(M)^\circ$ can be written as disjoint union of nonempty\footnote{In fact, one can show that $\mathcal{R}_0$ must be empty in this case, but this will not simplify the proof. We leave it for interested readers to check.} closed subsets, which is a contradiction to the fact that $\Phi(M)$ is a polytope.

    Now, without loss of generality, we may assume that there exists a connected component $\mathcal{C}$ of $\mathcal{R}$ such that $\overline{\mathcal{C}} \cap S_{i} \neq \emptyset$ for $i=1,2$. We show that there exist two facets $F,F'$ of $\Phi(M)$ that are also facets of $\overline{\mathcal{C}}$. By Theorem~\ref{thm:convexity-regular-values}, $\overline{\mathcal{C}}$ is a polytope, so $\partial \mathcal{C}$ is connected. By definition, $\partial \mathcal{C} \subset \partial \Phi(M) \cup \Phi(\Sigma_{n-1})$. Since $ \Phi(\Sigma_{n-1})$ is disconnected and $\partial \mathcal{C}$ intersects two of its connected components, 
    $\partial \mathcal{C} \cap \partial \Phi(M) \neq \emptyset$. Fix an $\alpha \in \partial \mathcal{C} \cap \partial \Phi(M)$. Since $\Phi(M)_{\short} = \partial \Phi(M)$, $\Phi^{-1}(\alpha)$ consists of a single orbit $o$. Let $Y= T \times_H \fh^\circ \times \C^{h+1}$ be the local model associated to $o$ with moment map $\Phi_Y([t,\nu,z]) = \frac{1}{2}\sum_{i=0}^h \eta_i |z_i|^2 + \nu$, where $\eta_0,\ldots,\eta_h$ are isotropy weights at any point in the orbit $o$. By Lemma~\ref{lem:local-cone}, there exists an open neighborhood $U$ of $\alpha$ such that $\Phi(M) \cap U = (\alpha+ \sum_{i=0}^h \R_{\geq 0} \eta_i + \fh^\circ) \cap U$. Hence, the facets of $\Phi(M)$ containing $\alpha$ are in one-to-one correspondence with the facets of the cone $\sum_{i=0}^h \R_{\geq 0} \eta_i$.

    On the other hand, by Corollary~\ref{cor:cx-1-rvlnf}, we can further assume that (after possibly relabelling the weights) $\overline{\mathcal{C}} \cap U = (\alpha+\sum_{i=0}^{h-1} \R_{\geq 0} \eta_i + \fh^\circ) \cap (\alpha+\sum_{i=1}^h \R_{\geq 0} \eta_i  + \fh^\circ) \cap U$ and $\xi_0 \xi_h <0$, where $\xi_0,\ldots,\xi_h$ are exponents of the defining monomial of $Y$. Hence, the facets of $\overline{\mathcal{C}}$ containing $\alpha$ are in one-to-one correspondence with the facets of $(\sum_{i=0}^{h-1} \R_{\geq 0} \eta_i) \cap (\sum_{i=1}^h \R_{\geq 0} \eta_i)$. We prove that $\sum_{i=1}^{h-1} \R_{\geq 0} \eta_i$ is a facet of both $(\sum_{i=0}^{h-1} \R_{\geq 0} \eta_i) \cap (\sum_{i=1}^h \R_{\geq 0} \eta_i)$ and $\sum_{i=0}^h \R_{\geq 0} \eta_i$. By Lemma~\ref{lem:maximally-short-weights}, $\mathcal{H} = \sum_{i=1}^{h-1}\R \eta_i$ is a hyperplane in $\fh^*$ and $\eta_0,\eta_h \notin \mathcal{H}$. Let $n$ be a normal vector of $\mathcal{H}$. Then for each $i \in \{1,\ldots,h-1\}$, $\langle \eta_i,n\rangle =0$. Since $\sum \xi_i \eta_i=0$ and $\xi_0 \neq 0$, $\langle \eta_0, n \rangle = -\frac{\xi_h}{\xi_0}\langle \eta_h, n \rangle$. Since $\xi_0 \xi_h <0$,  $\langle \eta_0, n \rangle$ and $\langle \eta_h, n \rangle$ are either both positive or both negative, so $\mathcal{H}$ is a supporting hyperplane of both $(\sum_{i=0}^{h-1} \R_{\geq 0} \eta_i) \cap (\sum_{i=1}^h \R_{\geq 0} \eta_i)$ and $\sum_{i=0}^h \R_{\geq 0} \eta_i$. Hence, there exist a facet $F_1$ of $\Phi(M)$, a facet $F_2$ of $\overline{\mathcal{C}}$, and a hyperplane $\mathcal{H}_\alpha \subset \ft^*$ such that $\mathcal{H}_\alpha \cap \Phi(M) = F_1$ and $\mathcal{H}_\alpha \cap \overline{\mathcal{C}} = F_2$. Since $ \overline{\mathcal{C}} \subset \Phi(M)$, we have $F_2 \subseteq F_1$. Since $F_1 \subset \partial \Phi(M)= \Phi(M)_{\short}$, by Corollary~\ref{cor:short-facet} $F_1=F_2$. 
    Hence, there exists a facet $F$ of $\Phi(M)$ that is also a facet of $\overline{\mathcal{C}}$. By \cite[Corollary 3.4]{sallee1967incidence}, $\partial \overline{\mathcal{C}} \setminus F$ is connected, so $(\partial \overline{\mathcal{C}} \setminus F) \cap \partial \Phi(M) \neq \emptyset$ and we can repeat the same argument to find another facet $F'$ of $\Phi(M)$ that is also a facet of $\overline{\mathcal{C}}$. 

    Let $\mathcal{H}_X(c) = \{\alpha \in \ft^*: \langle \alpha, X \rangle = c\}$ be the supporting hyperplane through $F$ with an inward-pointing normal vector $X$, i.e., $\Phi(M) \subset  \mathcal{H}_X^+(c):=\{\alpha \in \ft^*: \langle \alpha, X \rangle > c\}$. Similarly, let $\mathcal{H}_{X'}(c') = \{\alpha \in \ft^*: \langle \alpha, X' \rangle = c'\}$ be the supporting hyperplane through $F'$ with an inward-pointing normal vector $X'$, i.e. $\Phi(M) \subset  \mathcal{H}_{X'}^+(c'):=\{\alpha \in \ft^*: \langle \alpha, X' \rangle > c'\}$.
    Let $f_{\mathcal{C}}$ be the polynomial such that $f_{\mathcal{C}}|_{\mathcal{C}} = f|_{\mathcal{C}}$. By Lemma~\ref{lem:cx-1-wall-crossing} applied to $F$ and $F'$, there exist $A,A' >0$ such that 
    \[f_{\mathcal{C}} = \frac{\langle \cdot , X \rangle - c}{A} \quad \textrm{ and } \quad f_{\mathcal{C}} = \frac{\langle \cdot , X' \rangle - c'}{A'} .\]
    This implies $X = \frac{A}{A'}X'$ and $c = \frac{A}{A'}c'$, so $\mathcal{H}_X(c) = \mathcal{H}_{X'}(c')$, which is a contradiction.
    Therefore, we conclude that $\Phi(\Sigma_{n-1})$ must be connected.
\end{proof}

\begin{proof}[Proof of Theorem~\ref{thm:skeleton-connected}]
    Without loss of generality, we assume that there exists at least one top-dimensional tall exceptional orbit. Otherwise, this is just a restatement of Theorem~\ref{thm:n-1-skeleton}. Let $p$ be a tall exceptional point whose orbit has the same dimension as $T$. Let $N$ be the connected component of $M^H$ that contains $p$, where $H$ is the stabilizer group of $p$. By~\cite[Corollary 2.9]{Liu}, $N$ consists of exceptional points and $N$ is a connected, compact symplectic toric $(T/H)$-manifold whose Delzant polytope equals $\Phi(N)$. Let $q \in N^T$ be a $T$-fixed point in $N$. Since $N$ is a connected, $\Phi(N)$ is connected, so there exists a path connecting $\Phi(p)$ with $\Phi(q)$. Hence, the moment image of any tall exceptional point whose orbit has the same dimension as $T$ is connected to the moment image of a tall exceptional fixed point. By Theorem~\ref{thm:n-1-skeleton}, the moment images of any two tall exceptional fixed points are connected, so $\Phi(\Sigma)$ is connected.
\end{proof}

In~\cite{LPT}, we proved that the complexity one $T$-action on $(M,\om,\Phi)$ can be extended to a symplectic toric action if and only if the skeleton is $2$-colorable, the genus is $0$, and the painting is trivial. This immediately implies that if the skeleton is not $2$-colorable, then $\Ext(M,\om,\Phi)$ is empty, and if $\Ext(M,\om,\Phi)$ is nonempty, then the skeleton must be $2$-colorable. To understand the set of extensions, we need to understand $2$-colorable skeletons.

\begin{lemma}\label{lem:unique-2-coloring}
    Let $(M,\omega,\Phi)$ be a $2$-colorable complexity one $T$-space with proper moment map. Let $\Sigma$ be its skeleton such that $\Phi(\Sigma)$ is connected. If there exist nonempty clopen subsets $E_1,E_2,E_1',E_2'$ of $\Sigma$ such that $\Sigma = E_1 \sqcup E_2 = E_1'\sqcup E_2'$ and $\Phi$ restricts to an injection on each of $E_1,E_2,E_1',E_2'$, then either $E_1' = E_1$ or $E_1'=E_2$.
\end{lemma}

\begin{proof}
Let $E_{ij} = E_i \cap E_j'$. Then $\Sigma = E_{11} \sqcup E_{12} \sqcup E_{21} \sqcup E_{22}$.
For any orbit $o_1 \in E_{i1}$ and any orbit $o_2 \in E_{i2}$, since $\Phi$ is injective on $E_i$, $\Phi(o_1) \neq \Phi(o_2)$. Hence, $\Phi(E_{i1})$ is disjoint from $\Phi(E_{i2})$. Similarly, $\Phi(E_{1i}) $ is disjoint from $\Phi(E_{2i})$.  Hence, $\Phi(E_{11} \sqcup E_{22})$ is disjoint from $\Phi(E_{12} \sqcup E_{21})$. This implies that $\Phi(\Sigma) = \Phi(E_{11} \sqcup E_{22}) \sqcup  \Phi(E_{12} \sqcup E_{21})$. Since $\Phi$ is proper and $\ft^*$ is Hausdorff, $\Phi$ is a closed map. Hence, both $\Phi(E_{11} \sqcup E_{22}) $ and $ \Phi(E_{12} \sqcup E_{21})$ are closed as subsets of $\Phi(\Sigma)$. Since $\Phi(\Sigma)$ is connected, either $\Phi(E_{11} \sqcup E_{22})$ or $\Phi(E_{12} \sqcup E_{21})$ is empty.

If $\Phi(E_{11} \sqcup E_{22}) = \emptyset$, then $E_{11} = E_{22} = \emptyset$. It follows that $E_1 =E_{11} \sqcup E_{12} =E_{12} = E_1 \cap E_2'$. Hence, $E_1 \subset E_2'$. Similarly, $E_2' \subset E_1$. We conclude that $E_1 = E_2'$. If $\Phi(E_{12} \sqcup E_{21}) = \emptyset$, the same argument will show that $E_1 = E_1'$.
\end{proof}

Now, we are ready to prove the last main result Theorem~\ref{thm:max-short-bound}.

\begin{proof}[Proof of Theorem~\ref{thm:max-short-bound}]
    If $\Ext(M,\om,\Phi)$ is nonempty, then there exists a toric extension, so by~\cite[Lemma 1.2]{LPT} the skeleton is $2$-colorable. By Theorem~\ref{thm:skeleton-connected}, the moment image of the skeleton is connected, so we can apply Lemma~\ref{lem:unique-2-coloring} to conclude that there is a unique way to decompose the skeleton into two non-empty, disjoint clopen subsets so that the orbital moment map restricts to an injection on each of them. By~\cite[Theorem 2.6]{LPT}, for each decomposition, the set of extensions is isomorphic to $\ell \times \Z_2$.
\end{proof}

\bibliographystyle{alpha}
\bibliography{ref}

\end{document}